\documentclass[12pt]{amsart}
\usepackage{fancyhdr}
\usepackage{stmaryrd}
\usepackage{calrsfs}

\usepackage{amsmath}
\usepackage[nospace,noadjust]{cite}
\usepackage{amsfonts}
\usepackage{amssymb,enumerate}
\usepackage{amsthm}
\usepackage{cite}
\usepackage{comment}
\usepackage{color}
\usepackage[all]{xy}
\usepackage{hyperref}
\usepackage{lineno}
\usepackage{todonotes}

\newtheorem*{maintheorem*}{Main Theorem}
\newtheorem{theorem}{Theorem}[section]
\newtheorem{prop}[theorem]{Proposition}

\newtheorem{lemma}[theorem]{Lemma}

\theoremstyle{definition}

\newtheorem{example}[theorem]{Example}
\numberwithin{equation}{section}

\newcommand{\nn}{\mathbb{N}}

\newcommand{\qq}{\mathbb{Q}}
\newcommand{\rr}{\mathbb{R}}
\newcommand{\zz}{\mathbb{Z}}

\newcommand{\conv}{\mathsf{conv}}
\newcommand{\cone}{\mathsf{cone}}

\newcommand{\gp}{\text{gp}}

\newcommand{\norm}[1]{\left\lVert#1\right\rVert}
\newcommand{\ppp}{\mathsf{p}}

\newcommand{\uu}{\mathcal{U}}

\providecommand\ldb{\llbracket}
\providecommand\rdb{\rrbracket}

\hypersetup{pdftitle={FD}, colorlinks=true, linkcolor=blue, citecolor=cyan, urlcolor=wine}



\begin{document}
	
\mbox{}
\title{Subatomicity in Rank-$2$ Lattice Monoids}

\author{Caroline Liu}
\address{Department of Mathematics\\MIT\\Cambridge, MA 02139, USA}
\email{caroliu@mit.edu}

\author{Pedro Rodriguez}
\address{SMSS\\Clemson University\\Clemson, SC 29634, USA}
\email{pedror@clemson.edu}

\author{Marcos Tirador}
\address{Matcom, Universidad de La Habana, Plaza, Habana 10400, Cuba}
\email{marcosmath44@gmail.com}

%

\date{\today}

\begin{abstract}
	Let $M$ be a cancellative and commutative monoid (written additively). The monoid $M$ is atomic if every non-invertible element can be written as a sum of irreducible elements (often called atoms in the literature). Weaker versions of atomicity have been recently introduced and investigated, including the properties of being nearly atomic, almost atomic, quasi-atomic, and Furstenberg. In this paper, we investigate the atomic structure of lattice monoids, (i.e., submonoids of a finite-rank free abelian group), putting special emphasis on the four mentioned atomic properties.
\end{abstract}
\medskip

\maketitle


\bigskip
\section{Introduction}
\label{sec:intro}
\smallskip

Let $M$ be a cancellative and commutative monoid written additively. A non-invertible element $a$ in $M$ is an atom provided that whenever $a = u + v$ either $u$ or $v$ is an invertible element of $M$. A non-invertible element of $M$ is atomic if it can be written as a sum of atoms. Clearly, every atom is an atomic element. The monoid $M$ is atomic if every non-invertible element is atomic. Motivated by the landmark paper \cite{AAZ90} by D. D. Anderson, D. F. Anderson, and M. Zafrullah, the notion of atomicity has been systematically studied for the last three decades, with special focus on Krull monoids and the multiplicative monoids of integral domains.
\smallskip

Weaker versions of atomicity have been recently introduced and studied. In 2015, J. Boynton and J. Coykendall \cite{BC15} made known the notions of almost atomicity and quasi-atomicity. Following their paper, we say that $M$ is quasi-atomic (resp., almost atomic) if for each $b \in M$ there exists an element (resp., an atomic element) $m \in M$ such that the element $b+m$ is atomic. Observe that the properties of being quasi-atomic and being almost atomic are weaker versions of that of being atomic. Honoring the work of H. Furstenberg~\cite{hF55}, P. L. Clark in~\cite{pC17} presented the property of being Furstenberg in the context of integral domains. Following~\cite{pC17}, we say that $M$ is a Furstenberg monoid if every non-invertible element of $M$ is divisible by an atom. It is clear that the property of being Furstenberg is also a weaker version of that of being atomic. Most recently, in 2019, N. Lebowitz-Lockard \cite{nL19} studied all the weaker versions of atomicity mentioned in this paragraph, along with that of $M$ being nearly atomic, namely, the existence of $b \in M$ such that $b+m$ is atomic for all $m \in M$.
\smallskip

In the recent papers \cite{fG20a} and \cite{fG20}, F. Gotti studied certain arithmetic and factorization aspects of submonoids of finite-rank free abelian groups, which we call here lattice submonoids as they can be realized by monoids consisting of lattice points in $\zz^n$ for some $n \in \nn$. Motivated by the Frobenius coin problem, the rank-one lattice monoids have been largely considered in the literature under the name of numerical monoids. Numerical monoids are always atomic. However, this is not the case for lattice monoids with rank $2$ or higher. The primary purpose of this paper is to provide evidence of the complexity of the subatomic structure of lattice monoids of rank higher than $1$.
\smallskip

In Section~\ref{sec:background}, we introduce the terminology, notation, and main well-established results we will be using later. 
\smallskip

In Section~\ref{sec:overview}, after providing some initial examples illustrating how the size of the sets of atoms of additive submonoids of $\zz^n$ (specially $\zz^2$) can vary, we provide a sufficient condition for atomicity in the class consisting of all additive submonoids of $\zz^2$.  Examples illustrating different aspects of the atomicity and factorization of rank-$2$ submonoids have also been given in~\cite[Section~3]{fG20a} and more recently in~\cite[Sections~4 and~5]{GV23}. In addition, the set of atoms of the root-closed submonoids of $\zz^2$ has been fully described in~\cite{gL22}. Not long ago, an example of an atomic submonoid of $\zz^2$ that does not satisfy the ACCP was constructed in~\cite[Section~3]{fG23}.
\smallskip

In Section~\ref{sec:subatomicity}, we investigate the properties of being nearly atomic, almost atomic, and quasi-atomic. These three properties have been considered in~\cite{nL19} for the class of integral domains, in~\cite{GP23} for the class of semidomains (subsemirings of integral domains), and lately in~\cite{GV23} for the class of positive monoids (i.e., submonoids of a totally ordered group contained in its nonnegative cone). First, we quickly argue that every nearly atomic monoid is almost atomic (this follows as in the special case of integral domains, which was first argued in~\cite{nL19}). We also prove that in the class of rationally supported additive submonoids of $\zz^2$, the condition of being atomic and being nearly atomic are equivalent. Then, following a technique introduced in the proof of~\cite[Proposition~3.6]{fG23}, we construct a nearly atomic additive submonoid of $\zz^2$ that is not atomic. After that, we exhibit examples to confirm that near atomicity, almost atomicity, and quasi-atomicity are non-equivalent properties in the class consisting of all lattice monoids of $\zz^2$. Finally, we present a characterization of quasi-atomic monoids, which generalizes the results established in \cite[Theorem $5.2$]{GP23} and \cite[Theorem $8$]{nL19}.

\smallskip

Finally, in Section~\ref{sec:Furstenbergness}, we study the property of being Furstenberg. Besides being first investigated in the context of integral domains in~\cite{nL19}, the property of being Furstenberg has been considered in recent years in \cite[Section~5]{GL22} and \cite[Section~4]{GZ23}, also in the context of integral domains. The same property was most recently considered in~\cite[Section~3]{GP23} in the context of semidomains. Here we focus on the context of lattice monoids of $\zz^2$. First, we prove that for each $k \in \nn \cup \{\infty\}$ there exists a non-atomic Furstenberg monoid having precisely $k$ atoms. Then we show that every rationally supported lattice monoid of $\zz^2$ is a Furstenberg monoid.

\bigskip
\section{Background}
\label{sec:background}

In this section we introduce most of the relevant concepts related to commutative semigroups and convex geometry required to follow the results presented later. General references for any undefined term or notation can be found in~\cite{pG01} for commutative semigroups, in~\cite{GH06} for atomic monoids, and in~\cite{rtR70} for convex geometry. Also, see~\cite{GZ20} for a recent survey on factorizations in commutative monoids.

\medskip
\subsection{General Notation}

We set $\nn := \{1,2,\dots\}$ and $\nn_0 := \nn \cup \{0\}$. If $x,y \in \zz$, then we let $\ldb x,y \rdb$ denote the interval of integers between $x$ and $y$, i.e.,
\[
	\ldb x,y \rdb := \{z \in \zz \mid x \le z \le y\}.
\]
Clearly, $\ldb x,y \rdb$ is empty when $x > y$. In addition, for $X \subseteq \rr$ and $r \in \rr$, we set $X_{\ge r} := \{x \in X \mid x \ge r\}$, and we use the notation $X_{> r}$ in a similar way. Lastly, if $Y \subseteq \rr^d$ for some $d \in \nn $, then we set $Y^\bullet := Y \setminus \{0\}$.

\medskip
\subsection{Commutative Monoids}

Throughout this manuscript, the term \emph{monoid} refers to a cancellative, commutative semigroup with identity. Let $M$ be a monoid written additively. We set $M^\bullet := M \setminus \{0\}$, and we say that $M$ is \emph{trivial} if $M^\bullet = \emptyset$. The invertible elements of $M$ form a group, which we denote by $\uu(M)$, and $M$ is called \emph{reduced} if $\uu(M)$ is the trivial group. The \emph{difference group} $\gp(M)$ of $M$ is the unique abelian group $\gp(M)$ up to isomorphism satisfying that any abelian group containing an isomorphic image of $M$ also contains an isomorphic image of $\gp(M)$. The \emph{rank} of $M$ is taken to be the rank of $\gp(M)$ as a $\zz$-module, that is, the dimension of the vector space $\qq \otimes_\zz \gp(M)$. The monoid $M$ is \emph{torsion-free} if $nx = ny$ for some $n \in \nn$ and $x,y \in M$ implies that $x = y$. A monoid is torsion-free if and only if its difference group is torsion-free (see \cite[Section~2.A]{BG09}). The \emph{reduced monoid} of $M$ is the quotient of $M$ by $\uu(M)$, which is denoted by $M_{\text{red}}$. 
\smallskip

For $r,s \in M$, we say that $s$ \emph{divides} $r$ \emph{in} $M$ if there is a $t \in M$ such that $r = s + t$; in this case, we write $s \mid_M r$. A submonoid $M'$ of $M$ is called a \emph{divisor-closed submonoid} if every element of~$M$ dividing an element of~$M'$ in $M$ must belong to $M'$. If $S$ is a subset of $M$, then we let $\langle S \rangle$ denote the smallest submonoid of $M$ containing $S$, in which case, we say that $S$ is a \emph{generating set} of $\langle S \rangle$. The monoid $M$ is called \emph{finitely generated} provided that $M = \langle S \rangle$ for some finite subset $S$ of $M$.
\smallskip

An element $a \in M \setminus \uu(M)$ is called an \emph{atom} if whenever $a = r + s$ for some $r,s \in M$ either $r \in \uu(M)$ or $s \in \uu(M)$. We let $\mathcal{A}(M)$ denote the set consisting of all the atoms of $M$. Note that if~$M$ is reduced, then $\mathcal{A}(M)$ is contained in every generating set of~$M$.  A non-invertible element $b \in M$ is called \emph{atomic} (resp., \emph{almost atomic}, \emph{quasi-atomic}) provided that $b \in \langle \mathcal{A}(M) \rangle$ (resp., $b \in \gp(\langle \mathcal{A}(M) \rangle)$, $b \in \langle \mathcal{A}(M) \rangle - M := \{a - m \mid a \in \langle \mathcal{A}(M) \rangle \text{ and } m \in M\}$). Following~\cite{pC68}, we say that $M$ is \emph{atomic} if each non-invertible element of $M$ is atomic and, following~\cite{BC15}, we say that $M$ is \emph{almost atomic} (resp., \emph{quasi-atomic}) if each non-invertible element of $M$ is almost atomic (resp., quasi-atomic). It follows directly from the definitions that every atomic monoid is almost atomic and also that every almost atomic monoid is quasi-atomic. Following~\cite{nL19}, we say that $M$ is \emph{nearly atomic} if there exists $b \in M$ such that $b + m$ is atomic for all $m \in M$. It is clear that every atomic monoid is nearly atomic. In addition, every nearly atomic monoid is almost atomic: this was proved in~\cite{nL19} for the special case of integral domains, and we will see in Proposition~\ref{prop:NA implies AA} that it is also true in the more general case of monoids.

Following~\cite{pC17}, we say that the monoid $M$ is a \emph{Furstenberg monoid} if every element of $M \setminus \uu(M)$ is divisible by an atom in $M$. It follows from the definitions that every atomic monoid is Furstenberg. The converse does not hold in general: for instance, the monoid constructed in~\cite[Example~4.11]{GV23} is Furstenberg but not atomic (integral domains that are Furstenberg but not atomic also exist~\cite[Lemma~16]{nL19}). In addition, non-Furstenberg monoids with any prescribed number of atoms have been exhibited in~\cite{GZ23} (integral domains with nonempty set of atoms that are not Furstenberg also exist: see~\cite[Lemma~17]{nL19} for a stronger result and see~\cite[Section~5]{CG19} in tandem with \cite[Example~4.10]{GZ23} for the construction of a non-Furstenberg monoid domains with infinitely many atoms).

A subset $I$ of $M$ is called an \emph{ideal} of~$M$ if the set $I + M := \{r + s \mid r \in I \text{ and } s \in M\}$ is contained in~$I$ or, equivalently, if $I + M = I$. We say that an ideal $I$ of $M$ is proper if $I \neq M$. In addition, a proper ideal $P$ of $M$ is called \emph{prime} if whenever $r + s \in P$ for some $r, s \in M$, then either $r \in P$ or $s \in P$.

\medskip
\subsection{Factorizations}

A multiplicative monoid $F$ is said to be \emph{free on} a subset $A$ of $F$ provided that each element $x \in F$ can be written uniquely in the form
\[
	x = \prod_{a \in A} a^{\mathsf{v}_a(x)},
\]
where $\mathsf{v}_a(x) \in \nn_0$ and $\mathsf{v}_a(x) > 0$ only for finitely many $a \in A$. It is well known that for each set $A$, there exists a unique (up to isomorphism) monoid $F$ such that $F$ is a free monoid on $A$. For the monoid $M$, the free monoid on $\mathcal{A}(M_{\text{red}})$, denoted by $\mathsf{Z}(M)$, is called the \emph{factorization monoid} of $M$, and the elements of $\mathsf{Z}(M)$ are called \emph{factorizations}. If $z = a_1 \cdots a_n$ is a factorization in $\mathsf{Z}(M)$ for some $n \in \nn_0$ and $a_1, \dots, a_n \in \mathcal{A}(M_{\text{red}})$, then $n$ is called the \emph{length} of $z$ and is denoted by $|z|$. In addition, the unique monoid homomorphism $\phi \colon \mathsf{Z}(M) \to M_{\text{red}}$ satisfying that $\phi(a) = a$ for all $a \in \mathcal{A}(M_{\text{red}})$ is called the \emph{factorization homomorphism} of $M$. For each $x \in M$ the set
\[
	\mathsf{Z}(x) := \mathsf{Z}_M(x) := \phi^{-1}(x) \subseteq \mathsf{Z}(M)
\]
is called the \emph{set of factorizations} of $x$. Observe that $M$ is atomic if and only if $\mathsf{Z}(x)$ is nonempty for all $x \in M$ (notice that $\mathsf{Z}(0) = \{\emptyset\}$). 
For each $x \in M$, the \emph{set of lengths} of $x$ is defined by
\[
	\mathsf{L}(x) := \mathsf{L}_M(x) := \{|z| :  z \in \mathsf{Z}(x)\}.
\]
The sets of lengths of submonoids of $(\nn^d,+)$, where $d \in \nn$, have been considered in~\cite{fG20}. A survey on sets of lengths can be found in~\cite{aG16}. 

\medskip
\subsection{Euclidean Geometry and Convexity} 

Let $u$ be a nonzero vector in $\rr^2$, and let $L$ be a one-dimensional subspace (a line through the origin) of $\rr^2$. Then we let $\ppp_u$ and $\ppp_L$ denote the projection vectors on the one-dimensional subspace generated by $u$ and the one-dimensional subspace $L$, respectively. Also, we denote the upper  (resp., lower) closed half-space determined by $L$ as $L^+$ (resp., $L^{-}$). If $M$ is a submonoid of $\zz^2$, a \emph{supporting line} of $M$ is a line $L_M$ through the origin of $\rr^2$ such that either $M \subseteq L_M^+$ or $M \subseteq L_M^-$.  Additionally, if $M$ has a supporting line with rational slope, then we say that~$M$ is \textit{rationally supported}.  

For the rest of this subsection, fix $d \in \nn$ with $d \ge 2$. For any $x = (x_1, \dots, x_d)$ and $y = (y_1, \dots, y_d)$ in $\rr^d$, we let $\langle x, y \rangle$ denote the standard inner product of $x$ and $y$, that is, $\langle x, y \rangle = \sum_{i=1}^d x_i y_i$. In addition, for each $x \in \rr^d$, we let $\norm{x}$ denote the Euclidean norm of $x$. Also, if $x,y \in \rr^d$ and $Y$ is a nonempty subset of $\rr^d$, then we let $d(x,y)$ denote the Euclidean distance between $x$ and $y$, and we set
\[
	d(x,Y) := \inf \{ \norm{x-y} \, \mid y \in Y\},
\]
which is the Euclidean distance from $x$ to $Y$.

Let $S$ be a nonempty subset of $\rr^d$. The convex hull of $S$ (i.e., the intersection of all convex subsets of $\rr^d$ containing $S$) is denoted by $\conv(S)$. A nonempty convex subset~$C$ of $\rr^d$ is called a \emph{cone} if $C$ is closed under linear combinations with nonnegative coefficients. A cone $C$ is called \emph{pointed} if $C \cap -C = \{0\}$. The \emph{conic hull} of $S$, denoted by $\cone(S)$, is defined as follows:
\[
	\cone(S) := \bigg\{ \sum_{i=1}^n c_i s_i \ \bigg{|} \ n \in \nn, \ \text{and} \ s_i \in S \ \text{and} \ c_i \ge 0 \ \text{for every} \ i \in \ldb 1,n \rdb \bigg\},
\]
i.e., $\cone(S)$ is the smallest cone in $\rr^d$ containing $S$. For $s_1, \dots, s_k \in \rr^d$, we write $\textsf{cone}(s_1, \dots, s_k)$ instead of $\textsf{cone}(\{s_1, \dots, s_k\})$.

\bigskip
\section{Atomicity of Lattice Monoids}
\label{sec:overview}

\medskip
\subsection{Preliminary Examples of Lattice Monoids} 

The central objects of this paper are the lattice monoids. A \emph{lattice monoid} is a submonoid of a finite-rank free abelian group. Lattice monoids can be realized, up to isomorphism, as additive submonoids of $\zz^n$, where $n \in \nn$. The nontrivial lattice monoids of rank $1$ are the additive submonoids of $\zz$, which are either infinite cyclic groups or \emph{numerical monoids}, i.e., additive submonoids of $\nn_0$. Although it is easy to verify that every numerical monoid is finitely generated, this is not the case for lattice monoids of rank greater than $1$, as they are not finitely generated in general. This is illustrated in the following examples.

\begin{example}
	We proceed to present two atomic lattice monoids with infinitely many atoms, one of them having a reduced monoid with infinitely many atoms and one of them having a reduced monoid with exactly one atom.
	
	(1) Observe that the set $M := \{(0,0)\} \cup (\zz \times \nn)$ is closed under addition. Therefore, it is a lattice monoid. It is clear that $M$ is reduced. Figure \ref{fig:Example_infinite_many_atoms}(1) illustrates the elements of $M$. Now note that every lattice point contained in the line $y=1$ is an atom of $M$. In addition, one can readily see that every nonzero element of $M$ is the finite sum of lattice points in the line $y=1$. Hence, $\mathcal{A}(M) = \{(a,1) \mid a \in \zz\}$. Finally, since $M$ is reduced, it is isomorphic to its reduced monoid, and so $M_\text{red}$ also has infinitely many atoms.
	
	\smallskip
	
	(2) Now set $M := \zz \times \nn_0$, and observe that $M$ is a lattice monoid because both $\zz$ and $\nn_0$ are additive monoids. Figure \ref{fig:Example_infinite_many_atoms}(2) illustrates the elements of $M$. We can readily see that $\mathcal{U}(M) = \zz \times \{ 0 \}$. Therefore, for each $a \in \zz$, whenever $(a,1) = v_1 + v_2$ for some $v_1, v_2 \in M$ either $v_1$ or $v_2$ belongs to $\mathcal{U}(M)$. Thus, every element of the form $(a,1)$, where $a \in \zz$, is an atom of $M$, and so each element in $M \setminus \mathcal{U}(M)$ can be written as a sum of atoms. Hence, $M$ is atomic with $\mathcal{A}(M) = \{(a,1) \mid a \in \zz\}$. Since $a_1 + \mathcal{U}(M) = a_2 + \mathcal{U}(M)$ for all $a_1, a_2 \in \mathcal{A}(M)$, it follows that $\mathcal{A}(M_{\text{red}})$ is a singleton. Hence, $M_{\text{red}}$ is an atomic monoid with exactly one atom, from which we can infer that $M_{\text{red}}$ is isomorphic to the additive monoid $\nn_0$.
		\begin{figure}[h]
		\includegraphics[width=16cm]{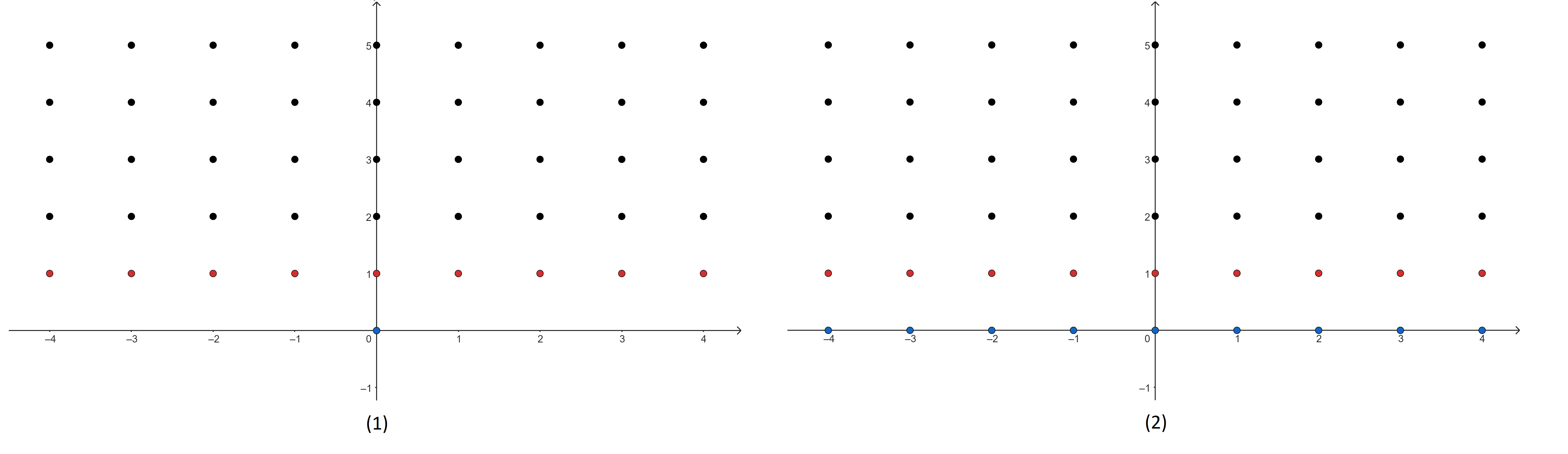}
		\caption{Elements of the two lattice monoids constructed in this example. Atoms are displayed in red and units in blue.}
		\label{fig:Example_infinite_many_atoms}
	\end{figure}
\end{example}

Indeed, there are lattice monoids having infinitely many atoms that are not atomic. The following example sheds some light upon this observation.

\begin{example}
	Consider the lattice submonoid $M' = \langle (0,0,1), (x,1,0) \mid x \in \zz \rangle$ of~$\zz^3$, and then set $M := M' \cup \{(x,y,z) \in \zz^3 \mid y \ge 2\}$. It is clear that $M$ is closed under addition. Therefore, $M$ is a lattice monoid. Since the intersection of $M$ with the plane determined by $y=0$ is the reduced monoid $M_0 := \{(0,0,z) \mid z \in \nn_0\}$, we conclude that $M$ is also a reduced monoid.
	
	We first show that $\mathcal{A}(M) = \{e_3, a_n \mid n \in \zz\}$, where $e_3 = (0,0,1)$ and $a_n = (n, 1,0)$ for every $n \in \zz$. Since $M_0$ is a divisor-closed submonoid of $M$, and $e_3$ is an atom of $M_0$, it follows that $e_3 \in \mathcal{A}(M)$. Now observe that $M$ intersects the plane determined by $y=1$ at the set $\{(x,1,z) \in \zz^3 \mid z \ge 0\}$. As none of the vectors in $M_0^\bullet$ divides any of the vectors in $\{a_n \mid n \in \zz\}$, it follows that $a_n \in \mathcal{A}(M)$ for every $n \in \zz$. Since $e_3$ divides $(x,1,z)$ for any $(x,z) \in \zz \times \nn$, it follows that $\mathcal{A}(M)$ intersects the plane determined by $y=1$ in $\{a_n \mid n \in \zz\}$. Finally, the fact that $e_3$ divides $(x,y,z)$ in $M$ for each $y \ge 2$ guarantees that $\mathcal{A}(M) = \{e_3, a_n \mid n \in \zz\}$, as desired.
	
	It is only left to verify that $M$ is not atomic. Indeed, since $\mathcal{A}(M)$ is contained in the half-space of $\rr^3$ determined by $z \ge 0$, the element $(0,2,-1) \in M$ cannot be written as a sum of atoms in $M$.
\end{example}

\medskip
\subsection{A Sufficient Condition for Atomicity}

There are non-atomic lattice submonoids of $\zz^2$ with infinitely many atoms, and we will construct some of such monoids in Example~\ref{ex:NA no AA}. However, the construction of those monoids requires more subtlety. This is given to the fact that if a non-atomic lattice monoid of rank~2 has a supporting line with rational slope, then it has only finitely many atoms. This is an immediate consequence of Theorem~\ref{thm:Nearly_rationa_finite_atoms}, which is the main result of this section and offers a sufficient condition for the atomicity of a lattice monoid. Before establishing that result, we need the following lemma.

\begin{lemma} \label{lemma1}
	Let $M$ be a lattice monoid of $\zz^2$. If $M$ is rationally supported, then~$M$ is isomorphic to a submonoid of $\mathbb{Z}  \times  \mathbb{N}_0$.
\end{lemma}

\begin{proof}
	Let $L$ be a supporting line of $M$ having a rational slope. Assume, without loss of generality, that $M \subset L^+$. Let $a/b$, with $a \in \zz$ and $b \in \nn$, be the slope of $L$. Now consider the function $f \colon M \to \mathbb{Z}  \times  \mathbb{N}_0$ defined as $f(r) = (\langle (b,a) , r \rangle, \langle (-a,b) , r \rangle)$. Observe that for every $r = (r_1, r_2) \in M$, it follows that $r_2 \ge \frac ab r_1$ and so $\langle (-a,b), r \rangle  = b \big( r_2 - \frac ab r_1 \big) \ge 0$. Thus, $f$ is well defined. In addition, the linearity of the maps $r \mapsto \langle (b,a), r \rangle$ and $r \mapsto \langle (-a,b), r \rangle$ ensures that $f$ is indeed a monoid homomorphism. To show that $f$ is injective take $r = (r_1, r_2)$ and $s = (s_1, s_2)$ in $M$ such that $f(r) = f(s)$. Then the equalities
	\[
		br_1 + ar_2 = bs_1 + as_2  \quad \text{ and }  - ar_1 + br_2 = - as_1 + bs_2 
	\]
	hold. This implies that $a(r_2 - s_2) = -b(r_1 - s_1)$ and $a(r_1 - s_1) = b(r_2 - s_2)$. After multiplying these equalities by $b$ and $a$, respectively, we see that
	\[
		a^2(r_1 - s_1)= ab(r_2 - s_2) = -b^2(r_1 - s_1).
	\]
	Therefore, $r_1 = s_1$, which implies that $r_2 = s_2$. Hence, the function $f$ is injective and, as a result, $M \cong f(M) \subseteq \mathbb{Z}  \times  \mathbb{N}_0$.
\end{proof}

In the following theorem we offer a sufficient condition for a lattice monoid to be atomic.

\begin{theorem} \label{thm:Nearly_rationa_finite_atoms}
	Let $M$ be a rationally supported lattice monoid. 
	If $|\mathcal{A}(M)| = \infty$, then $M$ is atomic. 
\end{theorem}

\begin{proof}
	In light of Lemma~\ref{lemma1}, we can assume that $M$ is a submonoid of $\zz \times \nn_0$. Now suppose, for the sake of a contradiction, that $M$ is not atomic. Let $b_0$ be a non-atomic element in $M$ that is not invertible. Then there exists a sequence $(b_n)_{n \in \nn_0}$ of non-invertible non-atomic and pairwise non-associate elements of $M$ such that $b_{n+1} \mid_{M} b_n$ for every $n \in \nn_0$. 
	For each $k \in \nn_0$, consider $L_k :=\{ (x,y) \in \zz^2 \mid y=k \}$. Let $u = (0, 1)$. We will split the rest of our proof into the following two cases.
	\smallskip
	
	\textit{Case 1:} $L_0 \cap M=\{(0,0)\}$ or $L_0 \cap M$ contains elements in both sides of the origin. In this case, it can be proved that every element in $L_0 \cap M$ is invertible. Hence, the sequence $(\norm{ \ppp_u(b_n) } )_{ n \in \nn_0}$ is strictly decreasing since its elements are non-associate. However, this sequence contains only finitely many different elements, which yields a contradiction.
	\smallskip
	
	\textit{Case 2:} $L_0 \cap M^\bullet$ contains elements in one and only one side of the origin. Let $(d_0,0)$ be the element with the smallest norm of $L_0 \cap M^\bullet$. Notice that there exist $N \in \nn_0$ and  $k_0 \in \nn_0$ such that for every $n \geq N$ the element $b_n$ belongs to $L_{k_0} \cap M$. 
	
	Now we prove that $m(d_0,0)$ divides the element $b_N$ for all $m \in \nn$. For every $n \geq N$ we have that $c_n := b_n - b_{n+1}$  is an atomic element of $L_0 \cap M$. Observe that $L_0 \cap M$ is a divisor-closed submonoid of $M$ that is isomorphic to a submonoid of $(\nn_0, +)$. Therefore, it has a finite number of atoms. Hence, at least one of them, namely $(d_1,0)$, divides an infinite number of $c_n$. Besides, the element $b_n - b_{n+1}$ divides $b_N$ for every $n \ge N$ since $b_N = b_N - b_{N+1} + b_{N+1} - b_{N+2}\cdots + b_n - b_{n+1} + b_{n+1}$. Then $m(d_1,0)$ divides the element $b_N$ for every $m \in \nn$. If $d_0 = d_1$ we are done. Suppose otherwise that $d_0 \ne d_1$ and let $m_0 = d_0 / \gcd(d_0, d_1)$. From B\'ezout's identity we know that there exist $r, s \in \nn_0$ such that  
	\[
		b_N - m(d_0, 0) = b_N - m m_0(\gcd(d_0,d_1), 0) = b_N - m m_0(r(d_1, 0) - s(d_0, 0)) \in M
	\]
	for every $m \in \nn$, as desired.
	
	For each $k \in \ldb 0, k_0 - 1 \rdb$, set $C_k:= \bigcup_{p \in \nn_0} L_{k+pk_0}$, and let us prove that $C_k$ contains at most finitely many atoms. Notice that $L_{k + pk_0} = \{x + p b_N \mid x \in L_k\}$ for every $p \in \nn_0$. Thus,
	\[
		C_k =  \{ x + p b_N \mid x \in L_k, p \in \nn_0 \} = \bigcup \limits_{d=0}^{d_0 - 1} \big \{ (d, k) + m(d_0,0) + p b_N \mid  m \in \zz, p \in \nn_0 \big\}. 
	\]
	Take $a_1, a_2 \in \{ (d,k) + m(d_0,0) + p b_N \mid  m \in \zz, p \in \nn_0 \big\}$ for some $d \in \ldb 0, d_0 - 1 \rdb$. Assume that $a_1$ and $a_2$ are both atoms and write $a_1 - a_2 = p' b_N +  m'(d_0, 0)$, for some $m', p' \in \zz$. Suppose, without loss of generality, that $p' > 0$. Therefore,
	\[
		a_1 - a_2 = p' b_N + m'(d_0, 0) = (p' - 1) b_N + (b_N + m' (d_0, 0)) \in  M,
	\]
	which contradicts that $a_1$ is an atom. Then $\{ (d,k) + m(d_0,0) + p b_N \mid  m \in \zz, p \in \nn_0 \big\}$ contains at most one atom for every $d \in \ldb 0, d_0 - 1 \rdb$. Hence, $C_k$ has at most $d_0$ atoms.  
	
	Finally, the equality $\zz \times \nn_0 =  C_0 \cup C_1 \cup \dots \cup C_{k_0-1}$  implies that
	\[
		|\mathcal{A}(M)|\le \sum \limits_{k = 0}^{k_0 - 1}|C_k \cap \mathcal{A}(M)| \le k_0 d_0,
	\]
	which is a contradiction.
\end{proof}

\bigskip
\section{Subatomicity}
\label{sec:subatomicity}

\medskip
\subsection{Near Atomicity}

In this subsection, we study near atomicity in lattice monoids. The following proposition, which is a generalization of \cite[Lemma~5]{nL19}, shows that every nearly atomic monoid is almost atomic.

\begin{prop} \label{prop:NA implies AA}
	Every nearly atomic monoid is almost atomic.
\end{prop}

\begin{proof}
	Let $M$ be a nearly atomic monoid. Then there exists some $b\in M$ such that for every $m \in M \setminus \mathcal{U}(M)$, the inclusion $b + m \in \langle \mathcal{A}(M) \rangle$ holds. Suppose first that $b$ is invertible. Then $m = b + (m - b) \in \langle \mathcal{A}(M) \rangle$ for all $m \in M$. Thus, $M$ is an atomic monoid, implying that $M$ is almost atomic. Now suppose that $b$ is not invertible. Notice that $2 b + m = b + (b + m) \in \langle \mathcal{A}(M) \rangle$ for all $m \in M$. This, along with the fact that $2 b \in \langle \mathcal{A}(M) \rangle$, shows that $M$ is almost atomic.
\end{proof}

Now we will prove that if a lattice monoid $M$ has a unique line $L$ as a supporting line and $L$ has an irrational slope, then if $M$ is a nearly atomic monoid then $M$ contains infinitely many atoms.

\begin{prop}
	Let $L$ be a line with irrational slope that is the only supporting line of a lattice monoid $M$. If $M$ is a nearly atomic monoid then $M$ contains infinitely many atoms.
\end{prop}

\begin{proof}
	Suppose that $M$ is nearly atomic. Let $v$ be a unit vector such that $L = v \rr$, and let $u$ be a unit vector that is normal to $v$ in a way that $\langle u,x \rangle > 0$ for all $x$ in $M$.  Suppose that there exists a vector $w \in \rr^2$ such that $\langle u,w \rangle > 0$, and for every $x \in M$ we have that $x \in \cone(-v,w)$. Hence, the line $w\rr$ is clearly a supporting line of $M$ which contradicts the uniqueness of $L$. Then for every vector $w$ satisfying that $\langle u,w \rangle > 0$, there exists some element $x \in M^\bullet$ such that $x \in \cone(w,v)$.  Hence, there exists some sequence $( d_n )_{n \in \nn}$, with $d_n \in M$ for every $n\in \nn$, such that $\lim\limits_{n \to \infty} \norm{d_n}=\infty$ and $d_n \in \cone(d_{n-1},v)$. Moreover, because $L$ has irrational slope it has no nonzero point with both coordinates rational, and therefore, $d_n$ is in the interior of $\cone(d_{n-1},v)$  for every $n\in \nn$. Then the angle between $d_n$ and $v$, namely $\alpha_n$, approaches to $0$ when $n$ goes large. On the other hand, since $M$ is nearly atomic, we know that there exists some $b\in M$ such that $b+d_n$ is an atomic element for every $n\in \nn$. Now let us prove that the angle between $b+d_n$ and $v$, namely $\beta_n$, approaches to $0$  when $n$ goes large. First observe that
	\[
		\lim \limits_{n\to \infty} \frac{ \langle v,d_n + b \rangle}{\norm{v} \norm{d_n}} = \lim\limits_{n\to \infty} \frac{\langle v,d_n \rangle}{\norm{v} \norm{d_n}} + \lim\limits_{n\to \infty}\frac{\langle v, b \rangle}{\norm{d_n}} = \lim\limits_{n\to \infty} \cos \alpha_n = 1,
	\]
	and also that
	\[
		\lim \limits_{n \to \infty} \frac{\norm{d_n + b}^2}{ \norm{d_n}^2}=\lim \limits_{n \to \infty} \frac{\norm{d_n}^2 + 2\langle d_n, b\rangle + \norm{b}^2}{\norm{d_n}^2} = 1 + \lim \limits_{n \to \infty} \left( 2 \frac {\norm{b} \cos \gamma_n}{\norm{d_n}} + \frac{\norm{b}^2}{\norm{d_n}^2}\right) = 1,
	\]
	where $\gamma_n$ is the angle between $b$ and $d_n$. Finally we have that
	\[
		\lim \limits_{n\to \infty} \cos \beta_n = \lim \limits_{n\to \infty} \frac{\langle v,d_n+b\rangle}{\norm{v}\norm{d_n+b}} = \lim\limits_{n\to \infty} \frac{\langle v, d_n + b \rangle}{\norm{v}\norm{d_n}} \lim \limits_{n \to \infty} \frac{\norm{d_n}}{\norm{d_n + b}} = 1,
	\]
	and so $\beta_n$ approaches to 0  when $n$ goes large and the sequence $(d_n+b)_{n\in\nn}$ has infinitely many elements in $\cone(w,v)$ for every $w$ such that $\langle u,w \rangle > 0$. So, if we suppose that $M$ has finitely many atoms, then there exists an atom $a$ such that all the atomic elements of $M$ belong to $\cone(-v, a)$ which contradicts the previous statement.
\end{proof}

We proceed to identify a large class consisting of lattice monoids where being nearly atomic and being atomic are equivalent conditions.

\begin{theorem}
	Let $M$ be a rationally supported lattice monoid of $\zz^2$. 
	 If $M$ is nearly atomic, then $M$ is atomic.
\end{theorem}

\begin{proof}
	Assume that $M$ is a rationally supported nearly atomic monoid. 
	 By virtue of  Lemma \ref{lemma1}, we can assume that $M$ is a submonoid of $\zz \times \nn_0$ and so $L$ is the line $y = 0$. Let $u = (0,1)$. Suppose, by way of contradiction, that $M$ is not atomic. Therefore, there exists a sequence $( d_n )_{n \in \nn}$ of pairwise non-associate elements of $M$ such that  $d_{n+1} \mid_M d_n$ for every $n \in \nn$ and none of its elements can be written as a sum of atoms.
	
	It follows from our argument in \emph{Case 1} of the proof of Theorem~\ref{thm:Nearly_rationa_finite_atoms} that $L$ contains an element of $M$ different from $(0,0)$. Moreover, we infer that $L$ does not contain elements of $M$ at both sides of the origin.

	Since $M$ is nearly atomic, we can take $r \in M$ such that $r + M \subseteq \langle \mathcal{A}(M) \rangle$. Observe that the sequence $( \norm{ \ppp_u(d_n) } )_{n \in \nn}$ is decreasing and has finitely many different elements. Then there exists $k \in \nn$ and $N \in \nn$ such that $\norm{ \ppp_u(d_n) } = k$ for all $n \geq N$. Also, notice that we can take the previous $N$ large enough to guarantee $ r + d_n$ is in the opposite side of $u\rr$ from $v$. This implies that $( \norm{ \ppp_L(r + d_{n}) } )_{n \in \nn}$ is strictly increasing for $n \geq N$.
	
Let us set $A := \{ a \in \mathcal{A}(M) \mid 1 \le  \norm{ \ppp_u(a) } \leq k + \norm{\ppp_u(r)} \}$. It follows from Theorem~\ref{thm:Nearly_rationa_finite_atoms} that $\mathcal{A}(M)$ has finitely many elements and so does $A$. Take $a$ to be the element in $A$ with maximum distance to $u\rr$. Since  $( \norm{ \ppp_L(r + d_{n}) } )_{n \in \nn}$ is strictly increasing for $n \geq N$, we can take $d_m$ with $m > N$ such that $\norm{\ppp_L(r + d_m)} > \norm{\ppp_L(a)} (\norm{\ppp_u(r)} + k )$. Let $z = \sum_{i = 1}^{m_1} a_i + \sum_{i = 1}^{m_2} b_i$ be a factorization of $r + d_m$  such that $a_i \in A$ for every $1 \leq i \leq m_1$ and $b_{i} \in L$ for every $1 \leq i \leq m_2$. Hence,
	\begin{align*}
		\norm{\ppp_u(r)} +  \norm{\ppp_u(d_m)} = \norm{\ppp_u(r + d_m)} 
		&=\norm{\ppp_u\left (\sum_{i = 1}^{m_1} a_i + \sum_{i = 1}^{m_2} b_i \right )} \\
		&= \sum_{i = 1}^{m_1} \norm{\ppp_u(a_i)} +  \sum_{i = 1}^{m_2} \norm{\ppp_u(b_i)} =  \sum_{i = 1}^{m_1} \norm{\ppp_u(a_i)} \ge m_1,
	\end{align*}
	which directly implies that $m_1 \le \norm{\ppp_u(r)} + k $ because $k =  \norm{\ppp_u(d_m)}$. This leads to
	\begin{align*}
		\norm{\ppp_L(r + d_m)} = \norm{\ppp_L \left ( \sum_{i = 1}^{m_1} a_i +  \sum_{i = 1}^{m_2} b_i \right )}
		&  = \norm{\ppp_L \left ( \sum_{i = 1}^{m_1} a_i \right)} - \sum_{i = 1}^{m_2} \norm{\ppp_L(b_i)} \\
		& \leq \sum_{i = 1}^{m_1} \norm{\ppp_L(a_i)} -  \sum_{i = 1}^{m_2} \norm{\ppp_L(b_i)} \\
		& \le  \sum_{i = 1}^{m_1} \norm{\ppp_L(a_i)} \\
		& \le m_1  \norm{\ppp_L(a)} \\
		& \le \norm{\ppp_L(a)} (\norm{\ppp_u(r)} + k ),
	\end{align*}
	where the second equality holds since $b_i$ and $r + d_m$ are in different quadrants. We arrive here to a contradiction, implying that $M$ is atomic.
\end{proof}

Our final goal in this subsection is to construct a nearly atomic lattice monoid that is not atomic. For such a purpose, we need the following lemmas. 

\begin{lemma} \label{lem:lines and lattice points}
	Let $L$ be a real line through the origin in $\rr^2$. Then for every $\epsilon > 0$ there exists $v \in \zz^2 \setminus \{(0,0)\}$ such that $d(v,L) < \epsilon$.
\end{lemma}

\begin{proof}
	Fix $\epsilon > 0$. If the slope of $L$ is a rational number, then it is clear that $L$ must contain a nonzero vector with integer coordinates and the statement of the lemma follows. Then we assume that the slope of $L$ is irrational.  Take $x \in \rr$ such that $w := (x,1) \in L$. Since the slope of $L$ is irrational, $x$ is also irrational. It is well known that the set $\{n x - \lfloor n x \rfloor \mid n \in \nn\}$ is dense in $[0,1]$. Then there exists $m \in \nn$ such that $m x - \lfloor m x \rfloor < \epsilon$. After setting $v := (\lfloor m x \rfloor, m) \in \zz^2$, we see that
	\[
		d(v,L) \le \norm{v - mw} = m x - \lfloor m x \rfloor < \epsilon.
	\]
\end{proof}

\begin{lemma} \label{lem:lines and lattice points 2}
	Let $L$ be a line in $\rr^2$ with irrational slope. Then for every $\epsilon > 0$ there exists $v \in \zz^2 \setminus \{(0,0)\}$ such that $d(v,L) < \epsilon$. Moreover, we can take $v$ with both coordinates even, or with both coordinates odd.
\end{lemma}

\begin{proof}
	Let $L_0$ be a line through the origin parallel to $L$ and $u$ be a vector normal to $L_0$. Assume, without loss of generality, that $L$ is above $L_0$. Fix $\epsilon > 0$. Take a lattice point $v_0$ in $L^+$. From Lemma~$\ref{lem:lines and lattice points}$ there exists a lattice point $w \in L_0^+$ such that $\norm{\ppp_u(w)} < \epsilon$.
	Set 
	\[
		v := v_0 - \big \lfloor \norm{\ppp_u(v_0) - d(L, L_0)} / \norm{\ppp_u(w)} \big \rfloor  w. 
	\]
	Then we have that $d(v, L) = \norm{\ppp_u(v)} - d(L, L_0) < \norm{\ppp_u(w)} < \epsilon$, which concludes the first part of our proof.
	
	Let $L'$ be a parallel line to $L$ such that $2 d(L', L_0) = d(L, L_0)$. From the first part of this proof, there exists $v \in \zz^2 \cap L'^+$ such that $ d(v,L') < \epsilon / 2$. Then we have
	\[
		\norm{\ppp_u(2v)} = 2\norm{\ppp_u(v)} < 2 (d(L', L_0) + \epsilon / 2) =  d(L, L_0) + \epsilon.
	\]
	Hence, $2v$ is a lattice point with even coordinates such that $d(2v, L) < \epsilon$.
	
	Finally, we must prove the existence of a lattice point $v'$ with odd coordinates such that $d(v', L) < \epsilon$. Let $w' = (-1,-1)$ and consider the affine line $L_{w'} := w' + L$. Then there exists a lattice point $v_1$ with even coordinates such that $d(v_1,L_{w'}) < \epsilon$. Observe that $v' := v_1 + (1,1)$ is a lattice point with odd coordinates. Hence, it follows from the definition of $L_{w'}$ that $d(v',L) < \epsilon$, as desired.
\end{proof}

We are in a position to exhibit a nearly atomic lattice monoid that is not atomic.

\begin{example} (cf. \cite[Proposition~3.6]{fG23}) \label{ex:NA no AA}
	Let $L$ be a line through the origin in $\rr^2$ with positive irrational slope. Let $v$ be the unit vector in the first quadrant such that $L = v \rr$, and let $u$ be the unit vector in the second quadrant that is normal to $v$. Also, for any vector $w$, let $L_w^+$ (resp., $L_w^-$) be the affine upper (resp., lower) closed half-space determined by the affine line $L_w := w + L$.
	
	Let $w = (0,2)$ and $\ell = \norm{\ppp_u(w)}$. Since the only lattice point of $L$ is the origin, the only lattice point of $L_w$ is $w$. Take $a_0$ to be a point with odd coordinates in $\cone(w,v)$ such that $\ell/2 < \norm{\ppp_u(a_0)} < \ell$. By Lemma~\ref{lem:lines and lattice points 2} we can take a  point $a_1$ with odd coordinates in the interior of $\cone(v, a_0)$ such that $\ell > \norm{\ppp_u(a_1)} > \norm{\ppp_u(a_0)}$. We can construct, by repeating this process, an infinite sequence of lattice points with odd coordinates $(a_n)_{n \in \nn_0}$ such that $a_{n + 1} \in \cone(v,a_n)$ and $\ell > \norm{\ppp_u(a_{n+1})} > \norm{\ppp_u(a_n)}$ for every $n \in \nn_0$.
	
	Now let $d_0$ be a point with even coordinates in the interior of $\cone(w,L_w)$  such that $\ell < \norm{\ppp_u(d_0)} < 2\norm{\ppp_u(a_0)}$. Following Lemma~\ref{lem:lines and lattice points 2} we can construct inductively a sequence of lattice points with even coordinates $(d_n)_{n \in \nn_0}$ such that $d_{n+1} \in \cone(d_n, L_w)$ and $\norm{\ppp_u(d_{n+1})} < \norm{\ppp_u(d_{n}) }$ for every $n\in \nn_0$.
	
	Let $A$ be the set of all lattice points with even coordinates that belong to $\cone(-v,u)$. Now consider the lattice monoid $ M := \langle A \cup \{a_n \mid n \in \nn_0\} \cup \{d_n\mid n \in \nn_0\} \cup \{w\} \rangle$. Figure~\ref{fig:nearly_no_atomic_lattice_monoid} shows the elements of this generator set of $M$, present in a region of $\zz^2$ containing the origin.
	
	We claim that all the terms of the sequence $(a_n)_{n \in \nn_0}$ are atoms in $M$. Assume otherwise, that $a_k$ is not an atom, for some $k \in \nn_0$. If it is the case, then $a_t$  divides $a_k$ for some $t\in \nn_0$, given that $a_k$ has odd coordinates.  Furthermore, $t < k$ holds, since $\norm{\ppp_u(a_t)} > \norm{\ppp_u(a_k)}$ for every $t > k$. Write $a_k = a_t + r$ for some $r \in M$. Then $r \in \cone(u,v)$ and $\norm{\ppp_u(r)} < \ell/2$, but this contradicts that $r \in M$.
	
	Now we aim to prove that $M$ is not atomic by arguing that none of the elements of $(d_n)_{n \in \nn_0 }$ can be written as a sum of atoms. First observe that $d_n$ is not an atom for any $n \in \nn_0$. This follows from the fact that $d_n - d_{n+1}$  is a vector that belongs to $\cone(u, -v)$ with even coordinates, for every $n \in \nn_0$, and hence, it belongs to $M$. Assume that there exists $m \in \nn_0$ such that $d_m$ can be written as a sum of atoms. Since $d_m \in \cone(u, v)$, there must be an atom $a$, also inside of $\cone(u,v)$, that divides $d_m$.  Suppose that $a = a_k$ for some $k \in \nn_0$. Since $a_k$ has odd coordinates, there exists $t \in \nn_0$ such that $ a_k + a_{t}$ divides $d_{m}$. However, $\norm{\ppp_u(d_m)} < 2\norm{\ppp_u(a_0)} < \norm{\ppp_u(a_k)} + \norm{\ppp_u(a_{t})}$, which leads to a contradiction. Since $a$ is inside $\cone(u, v)$ and it cannot be divided by any element of $\{a_n,d_n \mid n \in \nn_0\}$, the only possibility left is $a = w$. If it is the case, then $d_m - w$ lies inside of $\cone(u, v)$ and $\norm{\ppp_u(d_m - w)} < \norm{\ppp_u(a_0)}$. Thus, $d_m - w$  does not belong to $M$. We conclude that $d_m$ cannot be written as a sum of atoms.
	
	It is not hard to see that $w$ is also an atom. We claim now that every element in $A$ is atomic. Take an arbitrary $s \in A$ and let $r = s - \lfloor \norm{\ppp_u(s)} / \ell \rfloor  w$. Observe that $r$ is also in $A$, so $r \in M$. If we prove that $r$ can be written as a sum of atoms, then $s$ can also be written as a sum of atoms. Suppose that $a_k \mid_M r$ for some $k \in \nn_0$. Similar as we argue before this implies that $a_{k} + a_{t} \mid_M r$ for some $t \in \nn_0$, which is  not possible since $\norm{\ppp_u(r)} < \ell < \norm{\ppp_u(a_k + a_{t})}$. Since $w$ does not divide $r$ either, the divisors of $r$ are all in $\cone(-v,u)$. It is not hard to observe that $0$ is not a limit point of the projections in the $x$-axis of the elements in $\cone(-v,u)$, and thus $r$ can be written as a sum of atoms.
	
	Finally, let us prove that $M$ is nearly atomic by arguing that $w + M \subset \langle \mathcal{A}(M) \rangle$. Since every element in  $\langle A \cup \{a_n \mid n \in \nn_0\} \cup \{w\} \rangle $ is atomic, it is enough to show that every element $r \in \{w + d \mid d \in \langle \{d_n \mid n \in \nn_0\}\rangle \setminus (0,0) \}$ is atomic. Take $k \in \nn_0$ in such a way that $a_k$ is large enough to guarantee that $r - a_k \in \cone(-v, u)$. Since $\norm{\ppp_u(r)} > 2\ell$, we have that $r - 2a_k$ is also inside of $\cone(-v,u)$. Moreover, since $r$ has even coordinates, $r - 2a_k$ also has even coordinates, implying that $r - 2a_k \in A$. Then $r - 2a_k$ is atomic and so is $r$. Hence, we conclude that $M$ is a nearly atomic lattice monoid that is not atomic.
	
	\begin{figure}[h]
		\includegraphics[width=10cm]{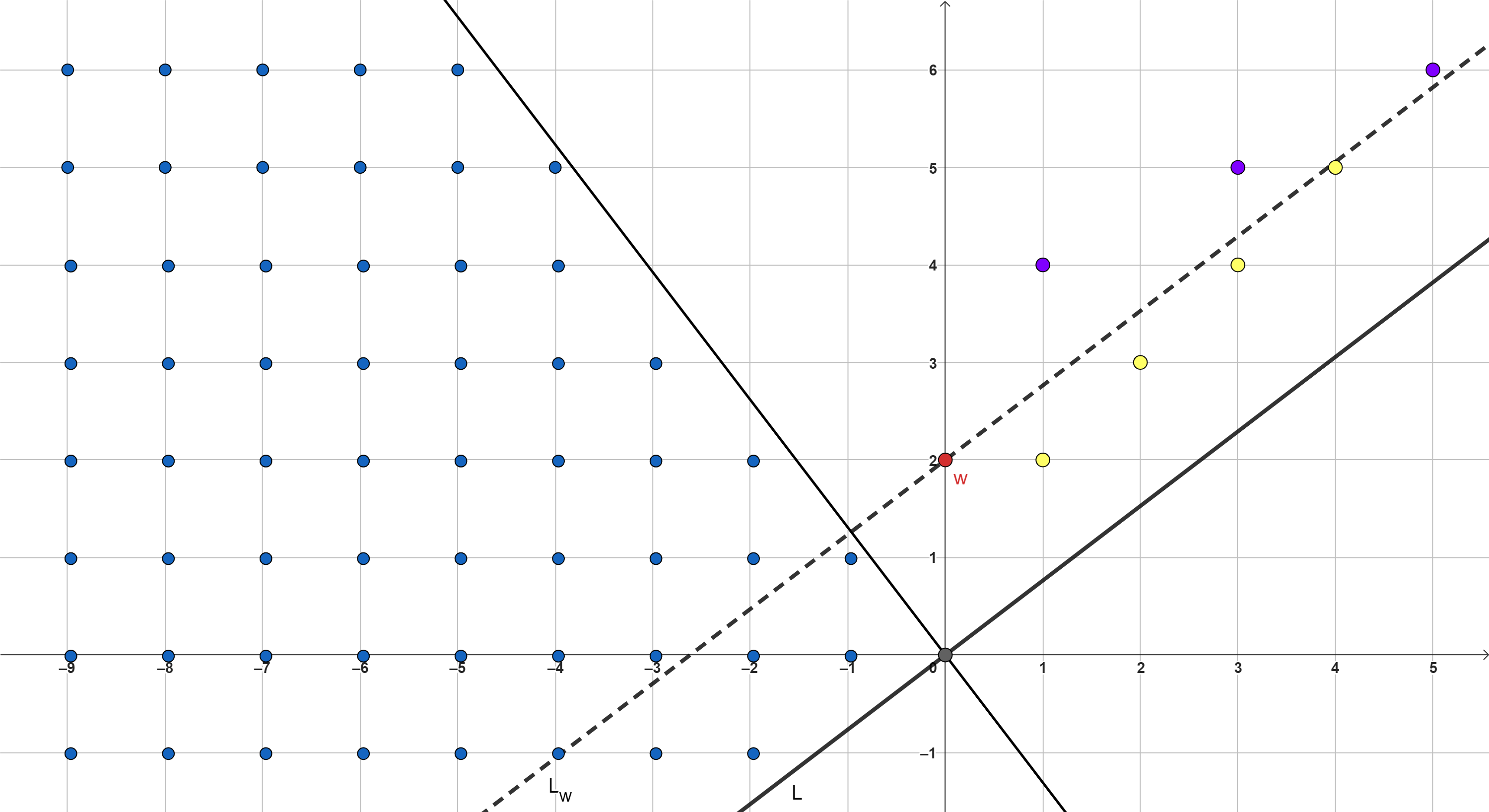}
		\caption{Region of $\zz^2$ showing some of the elements of the generator set used to define the lattice monoid in this example. The elements of $A$ are displayed in blue, those of $(d_n)_{n \in \nn_0}$ in purple and those of $(a_n)_{n \in \nn_0}$ in yellow. Also $w$ is displayed in red as one can see.}
		\label{fig:nearly_no_atomic_lattice_monoid}
	\end{figure}
\end{example}

\medskip
\subsection{Almost Atomicity and Quasi-Atomicity}

Let us take a look now at an example of an almost atomic lattice monoid that is not nearly atomic.

\begin{example}\label{ex:almost atomic not nearly atomic}
	Consider the lattice monoid $M'=\langle(1,0),(1,1) \rangle$. Now set $M := M'\cup \left(\zz \times \mathbb{N}_{\ge 2} \right)$, and observe that $M$ is closed under addition so it is a reduced lattice monoid. Figure~\ref{fig:almost_no_nearly} shows the monoid $M$. It is clear that $(1,0)$ and $(1,1)$ are both atoms. We claim that. $\mathcal{A}(M) = \{(1, 0), (1, 1)\}$. It suffices to show that none of the elements in $M \setminus M'$ is an atom of $M$. Fix $(x,y) \in M \setminus M'$. As $(x,y) \notin M'$, we see that $(x,y) \in \mathbb{Z}  \times\mathbb{N}_{\ge 2}  $, and so $(x - 1,y) \in \mathbb{Z}  \times\mathbb{N}_{\ge 2} $, which implies that $(1,0)$ divides $(x,y)$. Hence, $(x,y)$ is not an atom, as desired.
	\smallskip
	
	Let us show now that the monoid $M$ is almost atomic. Notice that $\{(x,y) \mid 0 \le y\leq x\} \subset \langle \mathcal{A}(M) \rangle$ (indeed, it is not hard to check that the other inclusion holds). Consider an element $(x,y) \in M \backslash \langle \mathcal{A}(M) \rangle$. Then $y>x$, which ensures that $(y-x,0) \in \langle (1,0) \rangle \subseteq \langle \mathcal{A}(M) \rangle$. Now we see that $(y-x,0) + (x,y) = (y,y) \in \langle \mathcal{A}(M) \rangle$. Hence, $M$ is almost atomic.
	
	Finally, we check that $M$ is not nearly atomic. Observe that for every element $(x,y) \in M$ the element $(-x,2)$ belongs to the monoid and $(-x,2) + (x,y) = (0,y+2) \notin \langle \mathcal{A}(M) \rangle$.
	
		\begin{figure}[h]
		\includegraphics[width=10cm]{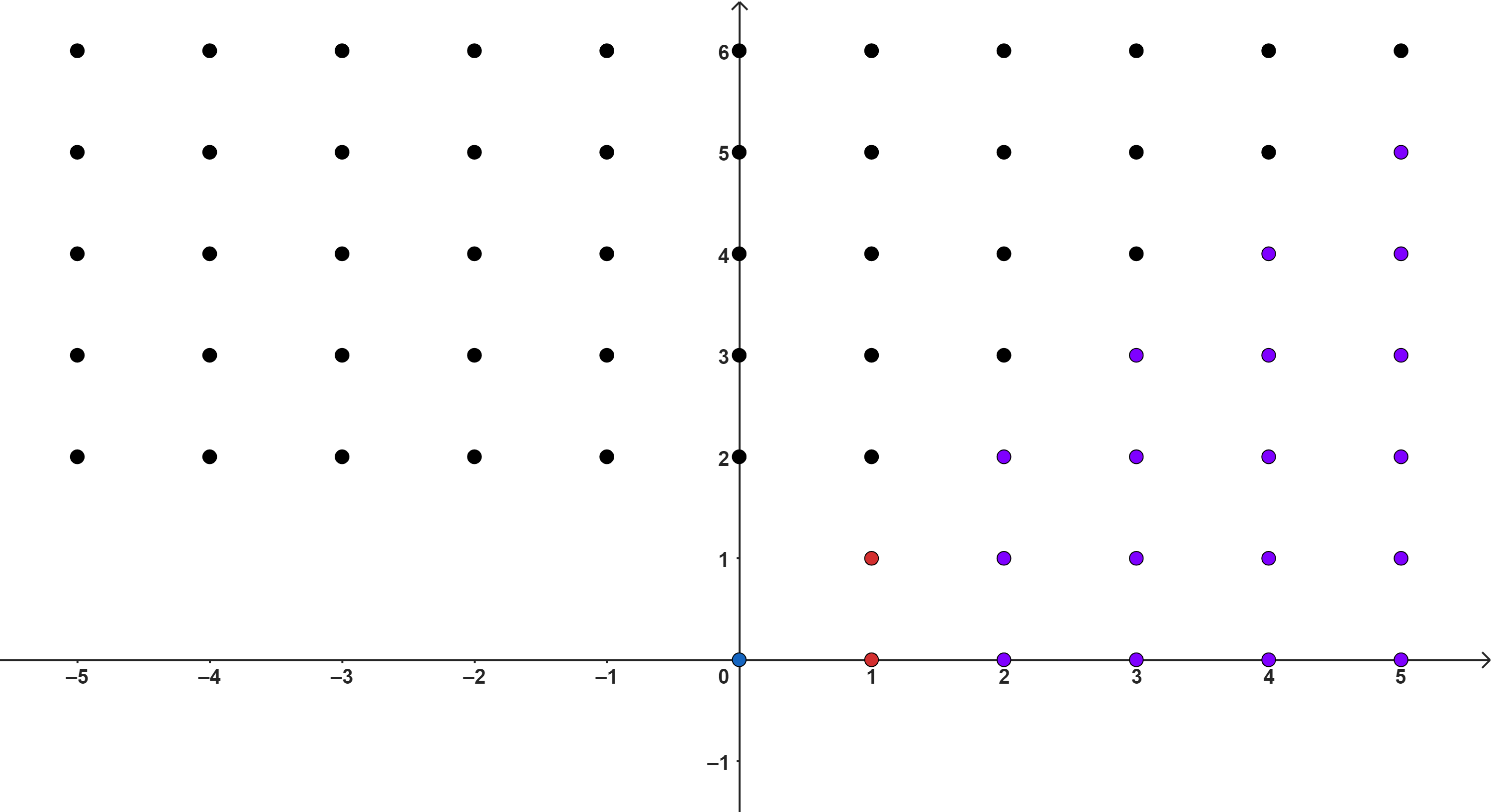}
		\caption{Almost atomic lattice monoid that is not nearly atomic. The atoms are displayed in red and the rest of the atomic elements in purple.}
		\label{fig:almost_no_nearly}
	\end{figure}
\end{example}

Here is an example of a quasi-atomic lattice monoid that is not almost atomic.

\begin{example} \label{ex:quasi-atomic not almost atomic}
	Consider the lattice monoid $M'=\langle(1,0),(0,2) \rangle$. Now set $M := M'\cup \left(  \mathbb{Z}  \times\mathbb{N}_{\ge 3} \right)$, and observe that $M$ is closed under addition so it is a lattice monoid. Figure~\ref{fig:quasi_no_almost} shows the monoid $M$. Also, note that $M$ is a reduced monoid. It is clear that $(1,0)$ and $(0,2)$ are both atoms. Indeed, $\mathcal{A}(M) = \{(1,0), (0,2)\}$, and this fact can be proved using an argument similar to that used in Example~\ref{ex:almost atomic not nearly atomic}.
	
	Let us show now that the monoid $M$ is quasi-atomic. Because $M' = \langle \mathcal{A}(M)\rangle$, it is enough to argue that each element of $M \setminus M' \subset  \mathbb{Z}  \times\mathbb{N}_{\ge 3} $ is a divisor of an element in $\langle \mathcal{A}(M) \rangle$. This is indeed the case, as for each $(x,y) \in M \setminus M'$, one immediately sees that $(-x,y) \in M$ and $(x,y)+(-x,y) =  (0,2y) \in \langle \mathcal{A}(M) \rangle$. Hence, $M$ is quasi-atomic.
	
	Finally, we shall verify that $M$ is not almost atomic. Since the second coordinate of each element in $\langle \mathcal{A}(M) \rangle$ is even, this amounts to observing that the second coordinate of $(0, 3) + (x,y)$ is odd for all $(x,y) \in \langle\mathcal{A}(M)\rangle$. Then we conclude that $M$ is a quasi-atomic lattice monoid that is not almost atomic.
\end{example}

		\begin{figure}[h]
			\includegraphics[width=10cm]{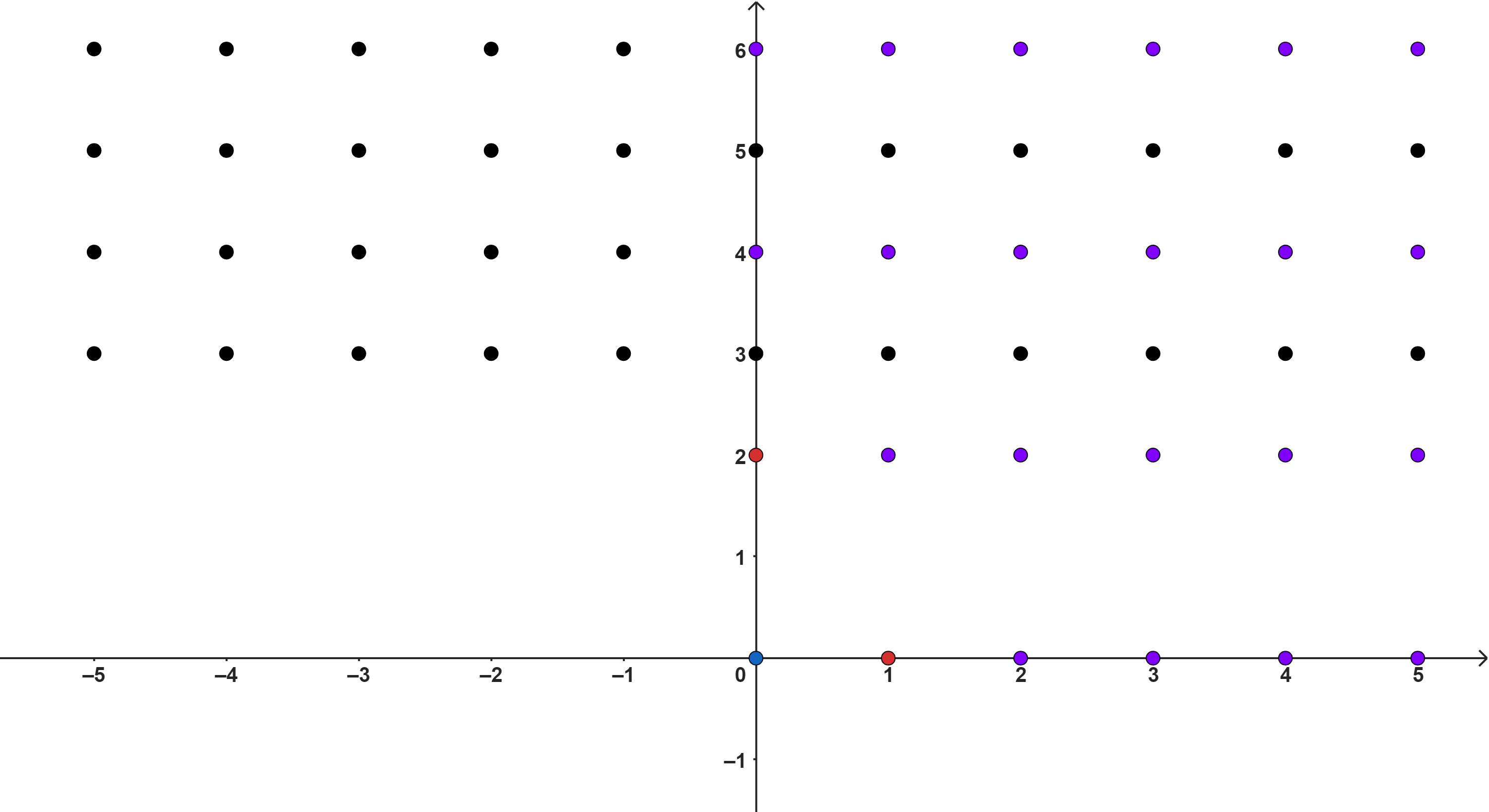}
			\caption{Quasi-atomic lattice monoid that is not almost atomic. The atoms are displayed in red and the rest of the atomic elements in purple.}
			\label{fig:quasi_no_almost}
		\end{figure}

 We conclude this section by presenting a characterization of quasi-atomic monoids, which is a generalization of \cite[Theorem $5.2$]{GP23}, which in turn is a generalization of \cite[Theorem $8$]{nL19}.

\begin{theorem}
	A monoid $M$ is quasi-atomic if and only if every nonzero prime ideal of~$M$ contains an atom.
\end{theorem}

\begin{proof}
	For the direct implication, suppose first that $M$ is quasi-atomic. Fix a nonzero prime ideal $P$ of $M$ and take an element $r \in P$. Because $M$ is quasi-atomic, there exists an element $s \in M$ such that $r + s$ can be written as a sum of atoms. Now since $r + s \in P$, we have that at least one atom in a factorization of $r + s$ belongs to $P$, as desired.
	
	For the converse, suppose that $M$ is not a quasi-atomic monoid. Then the set $P$ consisting of all the (non-invertible and) non-quasi-atomic elements of $M$ is nonempty. Let us argue first that $P$ is an ideal of $M$. Assume, otherwise, that there exist $r \in P$ and $s \in M$ such that $r + s \notin P$. As $r \in P$, it cannot be invertible, and so $r+s$ is not invertible. This, along with the fact that $r+s \notin P$, guarantees that $r+s$ is a quasi-atomic element of $M$. Then there exists $m \in M$ such that $m + (r + s) \in \langle \mathcal{A}(M) \rangle $. However, this implies that $r$ is also quasi-atomic, contradicting that $r \in P$. Hence $P$ is an ideal of $M$. Moreover, $P$ is a prime ideal since it is proper and the sum of two quasi-atomic elements is always quasi-atomic. Thus, $P$ is a nonzero prime ideal of $M$ that does not contain any atoms, which concludes our proof. 
\end{proof}

\bigskip
\section{Furstenbergness}
\label{sec:Furstenbergness}

Recall that a monoid $M$ is Furstenberg provided that every non-invertible element of $M$ is divisible by an atom. As the properties studied in the previous section, being Furstenberg is a property weaker than being atomic. The following realizability result shows the existence of a variety of Furstenberg lattice monoids that are not atomic.

\begin{prop}
	For each $k \in \nn \cup \{\infty\}$, there exists a non-atomic Furstenberg lattice monoid~$M$ of $\zz^2$ such that $|\mathcal{A}(M)| = k$.
\end{prop}

\begin{proof}
	Assume first that $k \in \nn$. Consider the numerical monoid $N_k := \{0\} \cup \zz_{\ge k}$. It is well known and easy to check that $\mathcal{A}(N_k) = \ldb k, 2k-1 \rdb$. Now take $M$ to be the lattice monoid $(N_k \times \{0\}) \cup (\zz \times \nn)$. Since $N_k \times \{0\}$ is a divisor-closed submonoid of $M$, it follows that $\ldb k, 2k-1 \rdb \times \{0\} \subseteq \mathcal{A}(M)$. Since every element of $M \setminus (N_k \times \{0\})$ is divisible by $(k,0)$, the equality $\mathcal{A}(M) = \ldb k, 2k-1 \rdb \times \{0\}$ actually holds. Hence, $|\mathcal{A}(M)| = k$. Since the second coordinate of each atom of $M$ is zero, it is clear that $M$ cannot be atomic. On the other hand, the atomicity of $N_k \times \{0\} \cong N_k$, along with the fact that the atom $(k,0)$ divides each element of $M \setminus (N_k \times \{0\})$, guarantees that $M$ is a Furstenberg monoid.
	\smallskip
	
	Let us now settle the case of $k = \infty$; that is, we should find a lattice monoid of $\zz^2$ with infinitely many atoms that is Furstenberg but not atomic. Consider the monoid~$M$ seen in Example \ref{ex:NA no AA} (Figure~\ref{fig:nearly_no_atomic_lattice_monoid}). We already proved that it is not atomic and contains infinitely many atoms. Let us prove that $M$ is a Furstenberg monoid. Since every element in $A \cup \{a_n \mid n \in \nn_0\}  \cup \{w\}$ is atomic, it only remains to show that every element in $\{d_n \mid n \in \nn_0 \}$ is divided by some atom. This follows immediately after observing that $d_{n} - d_{n+1} \in A$.
\end{proof}

\begin{prop}
	Let $M$ be a rationally supported lattice monoid of $\zz^2$. 
	Then $M$ is a Furstenberg monoid.
\end{prop}

\begin{proof}
	By virtue of Lemma~\ref{lemma1}, it suffices to prove that every additive submonoid of $\zz \times \nn_0$ is a Furstenberg monoid. Let $M$ be an additive submonoid of $\zz \times \nn_0$.
	
	We first argue that if $M \cap (\zz \times \{0\}) \subseteq \uu(M)$ (that is, $M \cap (\zz \times \{0\}) = \uu(M)$), then~$M$ is a BFM. To do so, suppose that $M \cap (\zz \times \{0\}) \subseteq \uu(M)$. If $M \subseteq \zz \times \{0\}$, then $M$ is a group, and so it is a BFM. Therefore, we can assume, without loss of generality, that $M$ is not a subset of $\zz \times \{0\}$, which implies that $M$ is not a group. Now let $u = (0,1)$. To argue that $M$ is a BFM, take a non-invertible element $b \in M$, and write $b = a_1 + \dots + a_\ell$ for some non-invertible elements $a_1, \dots, a_\ell \in M$. For each $i \in \ldb 1, \ell \rdb$, the fact that $a_i \notin \mathcal{U}(M)$ implies that $\norm{\textsf{p}_u(a_i)} \ge 1$, and so
	\[
		\ell \le \sum_{i=1}^\ell \norm{\mathsf{p}_u(a_i)} = \norm{\mathsf{p}_u \bigg( \sum_{i=1}^\ell a_i \bigg)} = \norm{\mathsf{p}_u(b)}.
	\]
	As a consequence, after assuming that we have chosen $\ell$ as large as it can possibly be, the maximality of $\ell$ will guarantee that the elements $a_1, \dots, a_\ell$ are atoms of $M$. Therefore, we obtain that $b$ is an atomic element such that $\max \mathsf{L}(b) \le \norm{\mathsf{p}_u(b)}$. Thus, we conclude that $M$ is a BFM.
	
	Finally, suppose, for the sake of a contradiction, that $M$ is not a Furstenberg monoid. Observe that $N := M \cap (\zz \times \{0\})$ is not a subset of $\uu(M)$ as, otherwise, $M$ would be a BFM by virtue of the argument given in the previous paragraph, contradicting that $M$ is not a Furstenberg monoid. Therefore, there is an element of $N$ that is not invertible, which means that $N$ is not a group. Since $N$ is isomorphic to a submonoid of $\zz$, and every submonoid of $\zz$ is either a group or a numerical monoid (up to isomorphism), we can further assume that either $N \subseteq \nn_0 \times \{0\}$ or $N \subseteq -\nn_0 \times \{0\}$. After replacing $M$ by its isomorphic copy $\{(-x,y) \mid (x,y) \in M\}$ if necessary, we can actually assume that $N \subseteq \nn_0 \times \{0\}$. Let $u$ be as in the previous paragraph, and among all the nonzero elements of $M$ that are not divisible by any atoms, let $v$ be one such that $\norm{\mathsf{p}_u(v)}$ is minimum. Since $v$ is not an atom, it can be written as $v = v_1 + v_2$ for some nonzero $v_1, v_2 \in M$ such that neither $v_1$ nor $v_2$ is divisible by an atom in $M$. Since $N$ is an atomic divisor-closed submonoid of $M$, it follows that every nonzero element of $N$ is divisible by an atom in $M$. As a result, neither $v_1$ nor $v_2$ are elements of $N$, which implies that $\norm{\mathsf{p}_u(v_1)} > 0$ and $\norm{\mathsf{p}_u(v_2)} > 0$. Thus, $v_1$ is an element in $M$ that is not divisible by any atoms and satisfies that $\norm{\mathsf{p}_u(v_1)} < \norm{\mathsf{p}_u(v)}$, which contradicts the minimality of $v$.
\end{proof}

Finally, we construct a lattice monoid of $\zz^2$ which is not a group and contains no atoms, and so it is not a Furstenberg monoid. Following the terminology introduced in~\cite{CDM99} by Coykendall, Dobbs, and Mullins (in the context of integral domains), we say that a monoid is \emph{antimatter} if its set of atoms is empty. It is clear, for instance, that the additive monoid $\qq_{\ge 0}$ is an antimatter monoid. Integral domains that are antimatter also exist, and they were first studied in~\cite{CDM99}. Furthermore, various class of antimatter monoid domains were identified in~\cite{ACHZ07}. We conclude this paper exhibiting an example of an antimatter lattice monoid of $\zz^2$.

\begin{example}
	Let $L$ be a line through $0$ in $\rr^2$ such that $L$ has negative irrational slope, and let $u$ be the unique unit vector in $\rr_{\ge 0} \times \rr_{\ge 0}$ that is orthogonal to $L$. For instance, we can take $L = \rr(-\frac{\sqrt{2}}2,\frac{\sqrt{2}}2)$ and $u = (\frac{\sqrt{2}}2, \frac{\sqrt{2}}2)$. Now consider the lattice monoid $M$ defined as follows:
	\[
		M := \{v \in \zz^2 \mid \langle v,u \rangle \ge 0 \}.
	\]
	We claim that $M$ contains no atoms. To argue this, fix $w \in M$. Lemma~\ref{lem:lines and lattice points} guarantees the existence of $v \in \zz^2 \setminus \{(0,0)\}$ such that $d(v, L) < d(w, L)$. We can assume, without loss of generality, that $v \in M$ (otherwise, replacing $v$ by $-v$). Therefore, we see that
	\[
		\langle w-v, u \rangle = \norm{u} \bigg( \frac{\langle w,u \rangle}{\norm{u}} - \frac{\langle v,u \rangle}{\norm{u}} \bigg) = \norm{u} \big( d(w,L) - d(v,L) \big) > 0.
	\]
	This implies that $w - v \in M$, and then we can decompose $w$ in $M$ as $w = v + (w - v)$. As the only point with rational coordinates in the line $L$ is $(0,0)$, the monoid $M$ must be reduced, and so the decomposition $w = v + (w - v)$ ensures that $w \notin \mathcal{A}(M)$. As a consequence, the monoid $M$ is antimatter, as desired.
\end{example}

\bigskip
\section*{Acknowledgments}

The authors are grateful to their mentor, Dr. Felix Gotti, for proposing the project and main questions motivating this paper and also for his guidance throughout the preparation of the same.

\bigskip

\end{document}